\documentclass[11pt,reqno]{amsart}
\usepackage{latexsym,amsxtra,amscd,ifthen,amsmath,color, multicol,hyperref,mathdots,mathrsfs}
\usepackage{amsfonts}
\usepackage{verbatim}
\usepackage{amsmath}
\usepackage{amsthm}
\usepackage{amssymb}
\usepackage[all,cmtip]{xy}
\usepackage[normalem]{ulem}
\usepackage{enumerate}
\usepackage[latin1]{inputenc}
\usepackage{amsmath}
\usepackage{amsfonts}
\usepackage{amssymb}
\usepackage{graphicx,mathtools}
\usepackage{multirow}
\usepackage{latexsym}
\usepackage{amssymb}
\allowdisplaybreaks
\numberwithin{equation}{section}

\theoremstyle{plain}
\newtheorem{thm}{Theorem}[section]
\newtheorem{lem}[thm]{Lemma}

\newtheorem{cor}[thm]{Corollary}

\theoremstyle{definition}

\newtheorem{question}[thm]{Question}

\newcommand{\kk}{\Bbbk}
\theoremstyle{remark}

\newcommand*{\Hom}{\ensuremath{\text{\upshape Hom}}}

\newcommand*{\Ext}{\ensuremath{\text{\upshape Ext}}}

\newcommand*{\End}{\ensuremath{\text{\upshape End}}}

\newcommand*{\Aut}{\ensuremath{\text{\upshape Aut}}}

\def\dim{\operatorname{dim}}
\def\id{\operatorname{id}}
\def\ord{\operatorname{ord}}

\def\Tr{\operatorname{tr}}
\def\Irr{\operatorname{Irr}}
\def\o{\otimes}

\DeclarePairedDelimiterX\set[1]\lbrace\rbrace{#1}

\begin{document}
\thispagestyle{empty}

\title{Hopf algebras of prime dimension in positive characteristic}

\author{Siu-Hung Ng}
\author{Xingting Wang}

\address{Department of Mathematics\\
Louisiana State University\\
Baton Rouge, LA 70803-4918, USA}
\email{rng@math.lsu.edu}
\thanks{The first author was partially supported by the NSF}
\address{Department of Mathematics\\
Howard University\\
Washington, DC 20059, USA}
\email{xingting.wang@howard.edu}



\begin{abstract}
We prove that a Hopf algebra of prime dimension $p$ over an algebraically closed field, whose characteristic is equal to $p$, is either a group algebra or a restricted universal enveloping algebra. Moreover, we show that any Hopf algebra of prime dimension $p$ over a field of characteristic $q>0$ is commutative and cocommutative when $q=2$ or $p<4q$. This problem remains open in positive characteristic when $2<q<p/4$.
\end{abstract}

\maketitle

\section*{Introduction}
The purpose of this paper is to classify Hopf algebras of prime dimension in positive characteristic under the assumption that the characteristic of the base field equals the dimension of the Hopf algebra. In characteristic zero, possible Hopf structure in prime dimension was first conjectured by Kaplansky \cite{Kap}, and later proved by Zhu \cite{Zhu}, which can be summarized as the following celebrated theorem:
\begin{thm}\label{T:Z}
A Hopf algebra of prime dimension over an algebraically closed field of characteristic zero is a group algebra.
\end{thm}
We are interested in the classification of Hopf algebras of prime dimension in positive characteristic.
\begin{question}\label{Q}
Let $H$ be a Hopf algebra of prime dimension $p$ over an algebraically closed field of positive characteristic $q$. Are the following true?
\begin{enumerate}
\item When $q=p$, $H$ is either a group algebra or a restricted universal enveloping algebra.
\item When $q\neq p$, $H$ is a group algebra.
\end{enumerate}
\end{question}
It is worthy mentioning that, by using the famous Lifting Theorem established by Etingof and Gelaki in \cite{EGelaki}, the answer to Question \ref{Q}(2) is affirmative when $q>p$ (cf. \cite[Corollary 3.5]{EGelaki}). While searching for general approaches, we discover that some of the techniques developed in \cite{EG2, Ng}, which was used in characteristic zero to study Hopf algebras of dimension being a product of two distinct primes, can be adapted in positive characteristic for Question \ref{Q}. As an application, we are able to provide an affirmative answer to Question \ref{Q} in the case when $q=2$ or $p<4q$, and narrow the problem down to the situation when $2<q<p/4$.

This paper is organized as follows. We will review and prove some representation theoretic results for finite-dimensional Hopf algebras in Section 1. The readers are referred to \cite{Mo} for more details on some basic definitions and results for Hopf algebras. We present our main theorems in Section 2.
\newline

\section{Preliminary}\label{S:1}
Let $\kk$ be a ground field of arbitrary characteristic. We use $H$ to denote a finite-dimensional Hopf algebra over $\kk$ with comultiplication $\Delta: H\to H\otimes H$ and antipode $S: H\to H$. Its dual vector space $H^*$ has a Hopf algebra structure induced from that of $H$.

By an $H$-module we always mean a left module over $H$ as an associated algebra. For $H$-modules $V$ and $W$, we give the tensor product $V\otimes W$ an action of $H$ via $\Delta$. The \emph{left dual} $V^*$ of an $H$-module $V$ is the left $H$-module with the underlying space $V^*=\Hom_\kk(V ,\kk)$ and the $H$-action given by
\[
(hf)(v)=f\left(S(h)v\right),\quad  \text{for}\ h\in H, \ f\in V^*,\ v\in V.
\]
We can twist the action of an $H$-module $V$ by the square of the antipode to obtain another $H$-module $\!_{S^2}V$. More precisely, $\!_{S^2}V=V$ as vector spaces with the $S^2$-twisted action given by
\begin{align}\label{Twist}
h\cdot v=S^2(h)v,\quad \text{for}\  h\in H,\ v\in V.
\end{align}
One sees easily that $\!_{S^2}V \cong V^{**}$ as $H$-modules via the natural isomorphism $j: V \to V^{**}$ of vector spaces. The category of finite-dimensional $H$-modules is a \emph{finite tensor category} in the sense of \cite{EO}.

The following result was observed in the proofs of \cite[Theorem 2.3(b)]{Lorenz2} and \cite[Lemma 2.11]{EG2}, whose categorical explanation can be found in \cite[Theorem 2.16]{EO}. The same result has been explicitly stated in \cite[Lemma 1.3]{Ng} as well. We include its proof for the sake of completeness with emphasis on the ubiquitousness of the approach regardless of the characteristic of the base field $\kk$.

\begin{lem}\label{L:1}
Let $V$ be a projective module over a nonsemisimple Hopf algebra $H$. Then the trace of any $H$-module map $\phi: V\to {\!_{S^2}}V$ is zero.
\end{lem}
\begin{proof}
First we observe that the evaluation map $\text{ev}: V^*\otimes V\to \kk$ and the dual basis map $\text{db}: \kk\to V\otimes V^*$ given by $\text{db}(1)=\sum_{i} v_i\otimes v^i$ are all $H$-module maps, where $\{v_i\}$ is a basis for $V$ and $\{v^i\}$ is the dual basis for $V^*$. Next, the composition of $H$-module maps
\[
\xymatrix{
\kk\ar[r]^-{\text{db}}  & V\otimes V^* \ar[r]^-{\phi\otimes \id}  &  {\!_{S^2}}V\otimes V^* \ar[r]^-{j\otimes \id}  & V^{**}\otimes V^*\ar[r]^-{\text{ev}} &
\kk
}
\]
is a scalar multiplication given by $\text{ev}\circ (\phi\otimes \id)\circ \text{db}(1)=\sum_{i} v^i(\phi(v_i))=\text{tr}(\phi)$. Suppose $\text{tr}(\phi)\neq 0$. Without loss of generality we can take $\text{tr}(\phi)=1$. Then $\kk$ becomes a direct summand of $V\otimes V^*$ via the embedding $\text{db}$ and the projection $\text{ev}\circ (\phi\otimes \id)$. Inasmuch as $V\otimes V^*$ is projective, by the fundamental theorem of Hopf modules \cite[Theorem 1.9.4]{Mo}, we conclude that $\kk$ is projective as well. This forces $H$ to be semisimple (cf. \cite[\S 2.4]{Lorenz}), which is a contradiction to our assumption. Hence $\text{tr}(\phi)=0$.
\end{proof}

Recall that an $H$-module $V$ is \emph{indecomposable} in case it is nonzero and has no nontrivial direct summands. We say a ring is \emph{local} if it has a unique maximal one-sided ideal. If $V$ is indecomposable, then the endomorphism ring $\End_H(V)$ is local (cf. \cite{AF,P}). Also note that a finite-dimensional Hopf algebra $H$ is local if and only if the dual Hopf algebra $H^*$ is connected \cite[Remark 5.1.7]{Mo}.

From now on, we assume $\kk$ is algebraically closed, and denote by $\text{Irr}(H)$ a complete set of nonisomorphic simple $H$-modules. For any $V\in \text{Irr}(H)$, the \emph{projective cover} of $V$ is denoted by $P(V)$. Since $H$ is a Frobenius algebra, we have (cf. \cite[61.13]{CR})
\begin{equation}\label{Dec}
\dim H=\bigoplus_{V\in \text{Irr}(H)} \dim V \cdot \dim P(V).
\end{equation}
For any $H$-module $M$ and $V\in \text{Irr}(H)$, we define $[M:V]$ to be the multiplicity of $V$ appearing as a composition factor of $M$. Let $V,W\in \text{Irr}(H)$. Since $V^*\otimes P(W)$ is projective, it is a direct sum of indecomposable projective $H$-modules. More precisely,

\begin{align}\label{PM}
V^*\otimes P(W)= & \bigoplus_{U\in \mathrm{Irr}(H)} \dim \left(\Hom_H(V^*\otimes P(W),U)\right)\cdot P(U) \nonumber\\
=&\bigoplus_{U\in \mathrm{Irr}(H)} [V\otimes U:W]\cdot P(U).
\end{align}

The idea of the next lemma comes from \cite[Corollary 2.10]{EG2}.

\begin{lem}\label{L:2}
Suppose $\mathrm{ord}\,(S^2)=q^n$ ($n\ge 0$) for some prime $q$ with $q\nmid \dim H$. Then there exists some $V\in \mathrm{Irr}(H)$ satisfying $q\nmid \dim V$, $q\nmid \dim P(V)$ and $V^{**} \cong V$ as $H$-modules.
\end{lem}
\begin{proof}
It is clear that the finite subgroup $G:=\langle S^2\rangle \subset \Aut(H)$ acts on $\text{Irr}(H)$ by twisting the $H$-module structures as \eqref{Twist}. Denote by $\mathcal O_1,\mathcal O_2,\cdots,\mathcal O_\ell$ the corresponding $G$-orbits in $\text{Irr}(H)$. According to \eqref{Dec}, we get
\begin{align*}
\dim H=\sum_{1\le i\le \ell} |\mathcal O_i| \cdot \dim V_i \cdot \dim P(V_i),
\end{align*}
where $V_i$ is any representative in $\mathcal O_i$. Since $q\nmid \dim H$, there is some index $1\le r\le \ell$ such that $q\nmid |\mathcal O_r|\dim V_r\dim P(V_r)$. Since $|\mathcal O_r|=|G|/|\text{Stab}(V_r)|=q^i$ for some $i\ge 0$, we get $\mathcal O_r=\{V_r\}$. This implies that $V_r\cong \!_{S^2}(V_r)\cong (V_r)^{**}$. So $V_r$ satisfies all the properties in the statement.
\end{proof}

Our last lemma is derived from the proof of \cite[Lemma 2.2(ii) and Corollary 2.3(II)]{Ng}.

\begin{lem}\label{L:3} Let $H$ be a nonsemisimple Hopf algebra over an algebraically closed field $\kk$ of characteristic  $q >2$ with $S^4=\id$. If $V$ is an indecomposable projective $H$-module of odd dimension and $V\cong V^{**}$ as $H$-modules,  then $\dim V\ge q$.
\end{lem}
\begin{proof}
Let $\phi: V \xrightarrow{\cong} V^{**}\xrightarrow{\cong} {\!_{S^2}} V$ be such an isomorphism of $H$-modules. Then $\phi^{n}\in \End_H(V)$ where $n=\ord\,(S^2)$. Since $V$ is an indecomposable projective $H$-module, $\End_H(V)$ is local (cf. \cite{P}). Hence, by rescaling $\phi$, we can assume $\phi^n$ to be unipotent. Moreover, since $n\le 2$, the eigenvalues of $\phi$ can only be $\pm 1$. Let $d_{\pm}$ be the corresponding dimensions of the $\pm 1$ generalized eigenspaces of $\phi$. Then $d_++ d_-=\dim V$ and $d_+-d_-\equiv 0 \pmod{q}$ since $\text{tr}(\phi)=0$ by Lemma \ref{L:1}. So $\dim V = 2a + mq$ where $a=\min(d_-, d_+)$ and $m$ is a positive integer. So $\dim V\ge q$ for $\dim V$ is odd. This concludes the proof.
\end{proof}

\section{Main results}\label{S:2}
Now we are able to provide an affirmative answer to address Question \ref{Q} in the case when $q=2$ or $p<4q$.
\begin{thm}\label{T:Main}
Let $H$ be a Hopf algebra of prime dimension $p$ over an algebraically closed field $\kk$ with characteristic $p$. Then $H$ is isomorphic to one of the following:
\begin{enumerate}
\item $\kk[C_p]$;
\item $\kk[x]/(x^p)$ with $\Delta(x)=x\otimes 1+1\otimes x$;
\item $\kk[x]/(x^p -x)$ with $\Delta(x)=x\otimes 1+1\otimes x$.
\end{enumerate}
\end{thm}
\begin{proof}
By \cite[Corollary 3.2(i)]{EGelaki}, $H$ cannot be both semisimple and cosemisimple since $\dim H=0$ in $\kk$.  We can simply assume $H$ to be nonsemisimple by duality. We first show that $H$ is  cocommutative or equivalently $H^*$ is commutative. If $p=2$, it is clear that the 2-dimensional Hopf algebra $H^*$, which is spanned by $1$ and another element $x\in H^*$, is commutative. Now let $p>2$.

We observe that the proof of $S^4=\id$ given in \cite[Theorem 2]{Zhu} remains true in positive characteristic. If $H$ has a nontrivial group-like element $g$, then $g$ must generate $H$ by the Nichols-Zoeller Theorem \cite{NZ}, and thus $H\cong  \kk[C_p]$. Therefore, $S^2=\id$. By the same argument, if $H^*$ has a nontrivial group-like element,  $S^2=\id$. If neither $H$ nor $H^*$ has nontrivial group-like elements, we obtain $S^4=\id$ from the celebrated Radford's formula \cite{Rad} of $S^4$.

Now we have $\ord(S^2)\le 2$ and  $\dim H=p>2$. Applying Lemma \ref{L:2}, there exists an indecomposable projective $H$-module $M$ of odd dimension such that $M \cong M^{**}$. Hence $M=H$ since $\dim M\ge p=\dim H$ by Lemma \ref{L:3}. So $H \cong \End_H(H)=\End_H(M)$ is local and $H^*$ is connected. It follows from the classification of all connected $p$-dimensional Hopf algebras \cite[Theorem 7.1]{WangP2} that $H^*$ is commutative.

Finally, since $H$ is always cocommutative and $\kk$ is algebraically closed, we know $H$ is pointed (cf. \cite[\S 5.6]{Mo}). If the coradical of $H$ contains a nontrivial group-like element, as discussed above, $H$ is a group algebra as in case (1). Otherwise, $H$ has trivial coradical, namely, $H$ is connected. Then $H$ is a restricted universal enveloping algebra as in cases (2) \& (3) (cf. \cite[Theorem 7.1]{WangP2}).
\end{proof}

We want to make a final remark regarding Question \ref{Q}(2) that our technique presented in this paper will generally fail if $q<p$. But when the ratio $q/p$ between the two distinct primes is greater than $1/4$, we can still provide a positive answer by pushing our method a little further. We need a few lemmas before establishing our second main result.

\begin{lem} \label{l:crucial}
Let $H$ be a nonsemisimple Hopf algebra of prime dimension $p$ over an algebraically closed field $\kk$ of characteristic $q \ne p$. Then $G(H)$ and $G(H^*)$ are trivial and there exists a simple $H$-module $M$ satisfying $\dim P(M) > \dim M>1$.
\end{lem}
\begin{proof}
  The case when $q=0$ follows from the same proof of \cite[Lemma 2.1(vi)]{Ng}. We now assume $q \ne 0$. If $H^*$ or $H$ has nontrivial group-like element, then, by the Nichols-Zoeller  Theorem,  $H^*$ or $H$ is isomorphic to a group algebra over $\kk$. In particular, $H$ is semisimple and cosemisimple for $q\neq p$ which contradicts the assumption of $H$.

  Since $H$ is nonsemisimple, $P(\kk) \ne \kk$ and hence $JP(\kk)/J^2P(\kk) \ne 0$, where $J$ is the Jacobson radical of $H$. Let $V$ be a simple $H$-submodule of $JP(\kk)/J^2P(\kk)$. Suppose $\dim V =1$. Then $V \cong \kk$ and so $\Ext(\kk,\kk) \ne 0$ (cf. \cite[Lemma 1.3]{CN}). This implies the existence of a nontrivial primitive element $\chi$ in $H^*$. The Hopf subalgebra of $H^*$ generated by $\chi$ can only be of dimension $q^l$ for some positive integer $l$, and $q^l \mid p$, a contradiction. This forces $\dim V > 1$ and $\Ext(\kk, V)\ne 0$.

  By taking dual, we find $\dim V^*>1$ and $\Ext(V^*, \kk)\ne 0$. Therefore, $\kk$ is submodule of $JP(V^*)/J^2P(V^*)$. Thus, $\dim P(V^*) \ge \dim V^*+1$.
\end{proof}

Next we consider the set
\begin{align}\label{Gamma}
\Gamma(H):=\left\{\left(\dim V,\dim P(V)\right)\in \mathbb N^2\,|\,V\in \mathrm{Irr(H)}\right\}.
\end{align}
of dimension pairs of $H$. We show in the following lemma that $|\Gamma(H)|>2$ and $|\Irr(H)| > 3$  by adapting the proof of  \cite[Proposition 3.6]{Ng}.
\begin{lem}\label{IrrNumber}
Let $H$ be a nonsemisimple Hopf algebra of prime dimension $p$ over an algebraically closed field $\kk$ of positive characteristic $q \ne p$.  Then $|\Gamma(H)|\ge 3$ and $|\mathrm{Irr}(H)|\ge 4$.
\end{lem}
\begin{proof}
It follows from Lemma \ref{l:crucial} that $|\Gamma(H)| \ge 2$. Suppose $|\Gamma(H)|=2$. Then $\Gamma(H)=\{(1,D_0),(d_1,D_1)\}$ where $D_1 > d_1 > 1$ (cf. Lemma \ref{l:crucial}). Let $V \in \Irr(H)$ such that $\dim V =d_1$ and $n$ the number of nonisomorphic simple $H$-modules of dimension $d_1$. Then, by \eqref{Dec}, we have
\begin{equation}\label{eq:1}
 p = D_0 + n d_1 D_1\,.
\end{equation}
Since the multiplicity of $P(\kk)$ in $V^* \o P(V)$ is 1, all other  indecomposable projective summands of $V^* \o P(V)$ are of dimension $D_1$ and so
\begin{equation}\label{eq:2}
d_1 D_1 = \dim(V^* \o P(V)) = D_0 + m D_1
\end{equation}
for some nonnegative integer $m$. These two equations \eqref{eq:1} and \eqref{eq:2} imply
$$
p = ((n+1)d_1 -m)D_1
$$
but this is impossible as $1< D_1 < p$ by \eqref{eq:1}. Therefore, $|\Gamma(H)| > 2$ and hence $|\Irr(H)|>2$.

Suppose $|\Irr(H)|=3$. Let $\Irr(H)=\{V_0,V_1,V_2\}$, where $V_0=\kk$, and $d_i=\dim V_i$ and $D_i = \dim P(V_i)$ for $i=0,1,2$. Since $G(H^*)$ is trivial by Lemma \ref{l:crucial}, it follows that $d_1, d_2 > 1$, and we may assume $d_1 < D_1$.  Also, note that $V_i^* \cong V_i$ for $i=1, 2$ for otherwise $V_1^* \cong V_2$ and hence $|\Gamma(H)|=2$. For $i,j,k =0,1,2$, write $N_{ij}^k=[V_i\otimes V_j:V_k]$ and we have
\begin{align*}
N_{ij}^k=[V_i\otimes V_j:V_k]=[(V_i\otimes V_j)^*:V_k^*]=[V_j^*\otimes V_i^*:V_k^*]=[V_j\otimes V_i:V_k]=N_{ji}^k.
\end{align*}
From \eqref{PM}, we get
\begin{align}\label{E1}
d_iD_k\,=\,\sum_{j=0}^2\, N_{ij}^k\, D_j\quad  \text{for}\quad 0\le i,k\le 2.
\end{align}
By \eqref{Dec},
\begin{align}\label{E1.5}
p=\sum_{j=0}^2 d_jD_j
\end{align}
which implies $1< D_1, D_2 < p$. Together with \eqref{E1}, we have
\begin{align}\label{E2}
p\,=\,\sum_{j=0}^2\, \left(d_j+\delta_{jk}d_i-N_{ij}^k\right)\, D_j\quad  \text{for all}\quad 0\le i,k\le 2.
\end{align}
By taking  $i=1,k=2$ in \eqref{E1} and $i=k=2$ in \eqref{E2}, we obtain
\begin{align}
N_{11}^2D_1&\,=(d_1-N_{12}^2)\,D_2=(d_1-N_{21}^2)\, D_2\,,\label{E3}\\
p&\,=(d_1-N_{21}^2)\,D_1+(2d_2-N_{22}^2)\,D_2.\label{E4}
\end{align}
By \eqref{E1}, $2d_2-N_{22}^2 \ge d_2 >1$. The equations \eqref{E3} and \eqref{E4} force $d_1-N_{21}^2 > 0$, otherwise $D_2$ is a nontrivial factor of $p$. Then, the equation \eqref{E4} further implies
$\gcd(D_1,D_2)=1$. So $D_1\,|\, d_1-N_{21}^2$ by \eqref{E3}. Since $D_1>d_1$, $d_1-N_{21}^2 = 0$, a contradiction! This completes our proof.
\end{proof}

\begin{thm}\label{P:M}
Let $H$ be a Hopf algebra of prime dimension $p$ over an algebraically closed field $\kk$ with positive characteristic $q < p$. If  $H$ is nonsemisimple, then  $q\ne 2$ and $p> 4q$. As a consequence, $H$ is isomorphic to a group algebra when $q =2$ or $ p< 4q$.
\end{thm}
\begin{proof}
Suppose $H$ is nonsemisimple. Then $S^2\neq \id$, otherwise $\text{tr}(S^2)=\dim H\neq 0$ in $\kk$, which implies that $H$ is semisimple by \cite[Proposition 2(c)]{Rad2}. By Lemma \ref{l:crucial}, $H$ and $H^*$ has no nontrivial group-like element. Therefore, $S^4=\id$. Since $S^2 \neq \id$, $\ord(S^2)=2$.

 If $q=2$, then the eigenvalues of $S^2$ can only be 1 and so $\Tr(S^2) = \dim H \ne 0 \pmod{2}$. Therefore, $q > 2$. By Lemma \ref{L:2}, there exists $V\in \Irr(H)$ such that both $V$ and $P(V)$ are of odd dimensions and $V^{**}\cong V$ as $H$-modules. Since $P(V)\cong P(V^{**})\cong P(V)^{**}$, $\dim P(V)\ge q$ by Lemma \ref{L:3}.

 Let $r=|\text{Irr}(H)|$. Then $r \ge 4$ by Lemma \ref{IrrNumber}. Recall from \eqref{PM} that $P(\kk)$ is a summand of $U^* \o P(U)$ for any $U \in \Irr(H)$. Therefore, $\dim U \cdot \dim P(U) \ge \dim P(\kk)$, and we get
 \begin{align*}
p\overset{\eqref{Dec}}{=}&\sum_{U\in \mathrm{Irr}(H)\setminus\{V\}} \dim (U^*\otimes P(U))+\dim V \cdot \dim P(V)\\\overset{\eqref{PM}}{\ge} & (r-1)\dim P(\kk)+\dim V\cdot \dim P(V) \geq
3 \dim P(\kk)+ q \dim V \,.
\end{align*}
Since $P(V)$ is a direct summand of $V \o P(\kk)$, we have $\dim V\cdot \dim P(\kk) \ge \dim P(V)$.  Since $\dim V$ is odd, the preceding inequality becomes
$$
p \ge q \cdot \left(\frac{3}{\dim V}+\dim V\right) \ge 4q\,.
$$
Since $p$ is a prime, $p > 4q$. This completes the proof of the first assertion.

If $q=2$ or $p<4q$, then $H$ and $H^*$ are semisimple by the first statement. Let $\mathcal O$ be the ring of Witt vectors of $\kk$ and $\mathcal K$ the field of fractions of $\mathcal O$. By the Lifting Theorem \cite[Theorem 2.1]{EGelaki}, we can construct a Hopf algebra $\overline{H}$ over $\mathcal O$ which is free of rank $p$ over $\mathcal O$ satisfying $H=\overline{H}/q\overline{H}$. Then $H_0=\overline{H}\otimes_\mathcal O \mathcal K$ is a $p$-dimensional Hopf algebra over $\mathcal K$. Since $\mathcal K$ is of characteristic 0, it follows from Theorem \ref{T:Z} that $\overline{H}\subset H_0$ are all cocommutative. Therefore, $H=\overline{H}/q\overline{H}$ is also cocommutative, and hence a group algebra since $\kk$ is algebraically closed.

\end{proof}

\begin{cor}
Let $H$ be a Hopf algebra of prime dimension $p$ over a field $\kk$ with positive characteristic $q$. Then $H$ is commutative and cocommutative when $q=2$ or $p<4q$.
\end{cor}
\begin{proof}
The case when $q>p$ is proved by Etingof and Gelaki in \cite[Corollary 3.5]{EGelaki}. It remains to verify for $q=2$ or $1/4<q/p\le 1$, which follows from Theorem \ref{T:Main} and Theorem \ref{P:M} by an algebraically closed field extension.
\end{proof}

\noindent\textbf{Acknowledgement.}
Part of this research work was done during the second author's visit to the Department of Mathematics at Louisiana State University in December 2015. He is grateful for the first author's invitation and wishes to thank Louisiana State University for its hospitality. The authors would also like to thank Pavel Etingof for his suggestion on improving the bound in Theorem \ref{P:M}, and the referee for his/her helpful comments.

\end{document}